\documentclass[oneside,english]{amsart}
\usepackage[T1]{fontenc}
\usepackage[latin9]{inputenc}
\usepackage{babel}
\usepackage{units}
\usepackage{amsthm}
\usepackage{amstext}
\usepackage{amssymb}
\usepackage{setspace}
\usepackage[unicode=true,
 bookmarks=true,bookmarksnumbered=false,bookmarksopen=false,
 breaklinks=false,pdfborder={0 0 1},backref=false,colorlinks=false]
 {hyperref}
\hypersetup{pdftitle={ A note on two Conjectures on Dimension funcitons of $C^{*}$-algebras},
 pdfauthor={Kaushika De Silva}}
\usepackage{breakurl}

\makeatletter
\numberwithin{equation}{section}
\numberwithin{figure}{section}
\theoremstyle{plain}
\newtheorem{thm}{\protect\theoremname}[section]
  \theoremstyle{plain}
  \newtheorem{conjecture}[thm]{\protect\conjecturename}
  \theoremstyle{definition}
  \newtheorem{defn}[thm]{\protect\definitionname}
  \theoremstyle{remark}
  \newtheorem*{rem*}{\protect\remarkname}
  \theoremstyle{plain}
  \newtheorem{lem}[thm]{\protect\lemmaname}
  \theoremstyle{remark}
  \newtheorem{rem}[thm]{\protect\remarkname}
  \theoremstyle{plain}
  \newtheorem{prop}[thm]{\protect\propositionname}

\makeatother

  \providecommand{\conjecturename}{Conjecture}
  \providecommand{\definitionname}{Definition}
  \providecommand{\lemmaname}{Lemma}
  \providecommand{\propositionname}{Proposition}
  \providecommand{\remarkname}{Remark}
\providecommand{\theoremname}{Theorem}

\begin{document}

\address{Department of Mathematics \\ Purdue University \\ 150 N. University St. \\ West Lafayette IN \\ USA \\ 47907 }
\email{pdesilva@math.purdue.edu}

\begin{onehalfspace}

\title{\noindent A note on two Conjectures on Dimension funcitons of $C^{*}$-algebras.}
\end{onehalfspace}

\author{\noindent Kaushika De Silva}
\begin{abstract}
\noindent Let $A$ be an arbitrary $C^*$ algebra. In \cite{BH} Blackadar and Handelman conjectured the set of lower semicontinuous dimension functions on $A$ to be pointwise dense in the set $DF(A)$ of all
dimension functions on $A$ and $DF(A)$ to be a Choquet simplex. We provide an equivalent 
condition for the first conjecture for unital $A$. Then by applying this condition we confirm the first Conjecture 
for all unital $A$ for which either the radius of comparison is finite or the semigroup $W(A)$
is almost unperforated. As far as we know the most general results on the first Conjecture
up to now assumes exactness, simplicity and moreover stronger regularity properties such as strict comparison. Our
results are achieved through applications of the techniques developed  in \cite{BR} and \cite{R}.

We also note that, whenever the first Conjecture holds for some unital $A$  and extreme boundary of the the quasitrace simplex of $A$ is finite, then every dimension function of $A$ is  lower semicontinuous and $DF(A)$ is affinely homeomorphic to the quasitrace simplex of $A$. Combing this with the said results on the first Conjecture give us a class of algebras for
which $DF(A)$ is a Choquet simplex, i.e. gives a new class for which the 2nd Conjecture mentioned above holds.
\end{abstract}

\maketitle

\section{\noindent Introduction.}

In \cite{C2} (c.f \cite{C1} ) Cuntz introduced the subequivalence relation $\preccurlyeq$ and used the relation to define dimension functions on (simple unital) $C^*$-algebras. Cuntz then associated the partially ordered abelian group $K_{0}^{*}(A)$  to a (simple unital) $C^*$-algebra $A$ and showed that dimension functions on $A$  bijectively correspond to states on $K_{0}^{*}(A)$, making available the methods of \cite{GH} to study dimension functions on $C^*$-algebras. As a group, $K_{0}^{*}(A)$ is the Grothendieck group of the Cuntz semigroup $W(A)$ which is the natural extension of the Murray-von Neumann semigroup of projections
$V(A)$ to positive elements in matrix algebras over $A$.

Continuing from \cite{C2}, Handelman \cite{Han} and later Blackdar and Handelman \cite{BH} developed a more general and a detailed theory for dimension functions on $C^*$-algebras. We focus on two Conjectures posted in \cite{BH};

\begin{conjecture}\cite{BH}\label{con1}
For any $C^{*}$- algebra $A$, the set of lower semicontinous dimension functions $LDF(A)$ is dense in $DF(A)$
in the topology of point wise convergence.
\end{conjecture}

\begin{conjecture}\cite{BH}\label{con2}
The affine space $DF(A)$ is a Choquet simplex for any $C^{*}$- algebra $A$.
\end{conjecture}

As shown in \cite{BH}, there is an affine and bijective natural map from the set of normalized 2-quasitraces $QT(A)$ of a $C^*$-algebra $A$ to the set $LDF(A)$, whose inverse is continuous. On the other hand Choquet simplexes are a natural extension of classical (finite) simplexes and the theory on Choquet simplexes is well developed - see \cite{G}. Thus, the Conjectures (if true), provide useful tools that can be applied to derive properties of $DF(A)$.

For non stably finite $C^*$-algebras the Conjectures hold trivially, as in this 
case ${K_0}^*(A)=0$ and $DF(A)$ is the
empty set. In the stably finite case there are several 
classes for which at least one of
the Conjectures are known to hold, as outlined below.

Conjecture 1.1 holds for unital commutative $A$ by \cite[Theorem I.2.4]{BH}. In \cite[Corollary 4.4]{Per}, Conjecture 1.2 is verified for unital (stably finite) $C^{*}$-algebras of real rank zero and stable rank one, providing the first (non trivial) examples for 1.2. The most general results on the Conjectures that we are aware of appear in \cite{BPT}. Theorem B of \cite{BPT} show that conclusions of both the Conjectures hold for unital, simple, separable, stably finite $C^{*}$- algebras which are either exact and $\mathcal{Z}$-stable 
or are $AH$-algebras of slow dimension growth. Furthermore \cite[Remark 6.5]{BPT} asserts that for exact $A$  Conjecture 1.1 holds assuming strict comparison instead of $\mathcal{Z}$-stability. Applying the methods in \cite{BPT}, 
several classes of continuous fields of $C^{*}$-algebras for which 1.1 and 1.2 hold are provided in \cite{ABPP}.

The above mentioned verifications have arisen more or less as applications of structure 
Theorems for $W(A)$ (i.e. \cite[Theorem 2.8]{Per} and \cite[Theorems 6.4 and 6.6]{BPT}). Apart 
form their usefulness in establishing the Conjectures these structure Theorems 
have other important applications - see \cite{BT}
for an example. However, when concerning the conjectures alone such Theorems 
are too strong requirements to ask for, at least if the Conjectures are to hold in full generality.

To our knowledge there has not been any work focused on the conjectures alone and this paper is an
attempt for a step in that direction. We aim to apply theory on state spaces of partially ordered semigroups developed mainly in \cite{BR} (c.f \cite{GH}) to study the conjectures. As it turns out this can be readily achieved, especially in the case of 1.1. 

In particular these techniques (of \cite{BR}) allow us to prove the following Theorem which give an alternate form of Conjecture 1.1 for unital $A$.

\noindent \textbf{Theorem 3.1}
Let $A$ be a unital $C^{*}$-algebra.
Then $LDF(A)$ is dense in $DF(A)$ if and only if $\iota:W(A)\to \mathit{LAff}_{b}(QT(A))^{+}$
is a stable order embedding.

By $\mathit{LAff}_{b}(QT(A))^{+}$ we mean the scaled partially ordered
abelian semigroup of bounded non negative lower semicontinuous
affine maps on $QT(A)$. $\iota$ is the natural map given by $\iota(\langle a\rangle)(\tau)=\lim_{n\to\infty}\tau(a^{\nicefrac{1}{n}}),\forall\tau\in QT(A)$. 
In a sense this is a weaker form of the representation of $W(A)$ given in \cite[Theorem 6.4]{BPT}.

Using the above we prove;

\noindent \textbf{Theorem 3.3} 
Let $A$ be any unital $C^{*}$-algebra. The following hold.

1. If $A$ has finite radius of comparison then $LDF(A)$ is dense in $DF(A)$.

2. If $W(A)$ almost unperforated then $LDF(A)$ is dense in $DF(A)$.

3. If $\partial_{e}(QT(A))$ is a finite set and if either of the
assumptions above (in 1,2) holds for $A$ then $DF(A)=LDF(A)$ and $DF(A)$
is affiinely homeomorphic to $QT(A)$. In particular $DF(A)$ is a
Choquet simplex.

To prove (1) and (2), we verify that the alternate form of Conjecture 1.1 provided in Theorem \ref{altBH} hold in the respective classes. In (1) this done by applying techniques of \cite{BR} once more while in the second case this is done by following the ideas of \cite{R}. Combining the  conclusions of parts 1 and 2 with Lemma \ref{lemc2} -  which mainly is
a consequence of Krein-Milman Theorem - we prove (3).

These results greatly extend the class of unital $C^*$-algebras for which the Conjectures (specially 1.1)
were known to hold. On the one hand these results do not assume simplicity or exactness as in \cite{BPT} and
on the other hand finite radius of comparison is a considerably weaker assumption than any 
of the regularity assumptions considered in \cite{BPT}. Most of the continuous fields considered
in \cite{ABPP} are also known to have finite radius of comparison.

In particular, the counter examples for Elliott's classification conjecture
constructed in \cite{tm} and Villadsen algebras of type I \cite{V1} have finite radius of comparison 
but are not covered by \cite{BPT}.  Furthermore Villadsen algebras of type II (\cite{V2}) are of finite
radius comparison and have unique quasitrace, and thus satisfy both the Conjectures from Theorem \ref{mainB}.
This means that for each $n\in\mathbb{N}$, we now know that there are unital algebras of stable rank $n$
which satisfy the Conjectures \ref{con1} and \ref{con2}.

For simple $C^*$-algebras, almost unperforation of $W(A)$ is equivalent
to strict comparison (i.e zero radius of comparison) and thus the second case may seem some what 
redundant when compared to 1. However, in general (without simplicity)
it is not clear how the two properties relate to each other. 

In the next section we recall some preliminary results and notations that we require. Section 3 contain the proofs of the main results.

\noindent \textbf{Acknowledgments}\textsl{.} I sincerely thank my adviser Prof. Andrew Toms for his guidance and for all encouragement provided throughout the project.

\section{Preliminaries and Notations}

\subsection{Partially ordered abelian semigroups.}

All semigroups we consider will be abelian. In addition we assume all semigroups to contain the identity element $0$.

\begin{defn}\label{p.o.a.s}
A partially ordered semigroup is a pair $(M,\leq)$ where
$M$ is a abelian semigroup and $\leq$ is a partial order on $M$ such that $\forall a,b,c\in M$, $a\leq b\implies a+c\leq b+c$. We also assume that $0\leq a$ for all $a\in M$.
\end{defn}

\begin{rem*}
The term partially ordered semigroup is
used even without assuming $0\leq a,\forall a\in M$
and the term positively ordered semigroup is used for
ones which in addition satisfy this. We do not have a need distinguish the two cases. 
\end{rem*}

All order relations we consider will be partial orders and for convenience we write ordered semigroup to mean a partially ordered semigroup in the sense of \ref{p.o.a.s}. 

An element $u$ in $(M,\leq)$ is called a order unit if for
each $x\in M$ there is some $n\in\mathbb{N}$ with $x\leq nu$.
A triple $(M,\leq,u)$ where $(M,\leq)$ and $u$ are as above is called a scaled ordered semigroup. If
the order and the order unit are clear we may write $M$ to denote $(M,\leq,u)$.

A morphism from $(M,\leq,u)$ to $(N,\leq,v)$
is a map $\phi:M\to N$ which is additive and order preserving with
$\phi(0)=0$ and $\phi(u)=v$.

A state on $(M,\leq,u)$ is a morphism from
$M$ to $(\mathbb{R}^+,\leq,1)$ where $\mathbb{R}^+$ is the additive
semigroup of non negative real numbers and $\leq$ is as usual. The
set of all states of $(M,\leq,u)$ will be denoted by $S(M,\leq,u)$ (or by $S(M)$
if the choice of order unit and order are clear). $S(M)$ is
compact and convex as a subset of the space of all real valued functions
on $M$ in the topology of pointwise convergence.

A morphism $\phi:(M,\leq,u)\to(N,\leq,v)$ as above induce a continuous
affine map $\phi^{\sharp}:S(N)\to S(M)$ via composition.

The following class of morphisms between scaled ordered semigroups was  introduced in \cite{BR}.

\begin{defn}\cite[Definition 2.2]{BR}\label{stoem} 
Let $(M,\leq,u)$,$(N,\leq,v)$ be scaled ordered semigroups and $\phi:M\to N$ be
a morphism of scaled ordered semigroups. $\phi$ is called a stable order embedding
if for any $x,y\in M$, there are $n\in\mathbb{N}$ and $z\in M$
with $nx+z+u\leq ny+z$ if and only if there are $m\in\mathbb{N}$
and $w\in N$ with $m\phi(x)+v+w\leq m\phi(y)+w$.
\end{defn}

We recall some useful results from \cite{BR}.

\begin{lem}\cite[Lemma 2.8]{BR}\label{domination} 
Let $(M,\leq,u)$ be a scaled
ordered semigroup and $x,y\in M$. Then $s(x)<s(y)$
for all $s\in S(M)$ if and only if there is some $n\in\mathbb{N}$
and $z\in M$ such that $nx+z+u\leq ny+z$. 
\end{lem}

\begin{lem}\cite[Lemma 2.9]{BR}\label{density}
Let $(M,\leq,u)$ be a scaled ordered semigroup
and $K$ be a nonempty compact convex subset $S(M)$. Suppose for
any $a,b\in M$ if $s(a)<s(b)$ for all $s\in K$ then $s(a)<s(b)$
for all $s\in S(M)$. Then $K=S(M)$. 
\end{lem} 

\begin{thm}\cite[Theorem 2.6]{BR}\label{thmstoem}
Let $\phi:(M,\leq,u)\to(N,\leq,v)$
be a morphism of scaled ordered semigroups. Then $\phi$ is a stable order
embedding iff 
\[
S(M)=\lbrace g\circ\phi:\, g\in S(N)\rbrace.
\]
\end{thm}

\begin{rem}
From \cite{BR} the above statements hold even when
$M,N$ are pre-ordered. As we will only be considering partially
ordered semigroups, we limit to this case. 

\end{rem}

\subsection*{Almost unperforation.}

\begin{defn}
An ordered semigroup $(M,\leq)$ is said to be almost unperforated if $kx\leq k^\prime y\implies x\leq y$ for all $x,y\in M$ and  each $k,k^\prime\in\mathbb{N}$ with $k>k^\prime$
\end{defn}

We will need the following Proposition which is proven in \cite{R} (c.f. \cite{B})

\begin{prop}\cite[Proposition 3.2]{R}\label{aup}
Let $(M,\leq)$  be almost unperforated and $u,x\in M$. If $u$ is a order unit and $s(x)<s(u)$ for all $s\in S(M,\leq,u)$ then $x< u$.
\end{prop}

\subsection{The Cuntz semigroup of a $C^{*}$-algebra and the group $K_{0}^{*}(A)$.}

\noindent As usual, let $W(A)$ denote the Cuntz semigroup of the $C^{*}$-algebra $A$.

That is, 
\[
W(A)=M_{\infty}(A)_{+}\diagup\sim
\]
where $\sim$ is the Cuntz equivalence relation. Recall that for $a,b\in M_{\infty}(A)_{+}$,
$a$ is said to be Cuntz equivalent to $b$ (written $a\sim b$) iff 
$a\preccurlyeq b$ and $b\preccurlyeq a$. $W(A)$  inherits the natural partial order given by
\[\langle a\rangle\leq\langle b\rangle\iff a\preccurlyeq b,
\]
and the pair $(W(A),\leq)$ form a ordered semigroup. 
If $A$ is unital then $\langle1_{A}\rangle$ is an order unit for $(W(A),\leq)$ and 
will always be the chosen order unit of $W(A)$ for us.

Following the notation of \cite{C2} let us write $K_{0}^{*}(A)$ to denote
the Grothendieck group of $W(A)$ and set 
\[
K_{0}^{*}(A)_{_{++}}=\left\{ \gamma(y)-\gamma(x):x,y\in W(A)\text{ and }x\leq y\right\} 
\]
where $\gamma:W(A)\to K_{0}^{*}(A)$ denotes the natural map given
by Grothendieck construction.

 From \cite{BH} (c.f \cite{C2}) the pair $(K_{0}^{*}(A),K_{0}^{*}(A)_{_{++}})$ form a partially ordered
abeilan group and $\gamma(\langle1_{A}\rangle)$ is a order unit
for $(K_{0}^{*}(A),K_{0}^{*}(A)_{_{++}})$. Note also that $(K_{0}^{*}(A),K_{0}^{*}(A)_{_{++}})$
is directed; i.e. $K_{0}^{*}(A)=K_{0}^{*}(A)_{_{++}}-K_{0}^{*}(A)_{_{++}}$. 

A state $s$ on $(G,G_{+},u)$ where $(G,G_{+})$ is a
partially ordered abelian group and $u$ is an order unit,
is an additive map $s:G\to\mathbb{R}$ satisfying $s(G_{+})\subset[0,\infty)$
and $s(u)=1$. Set of all states on $(G,G_{+},u)$ is denoted by $S(G,G_{+},u)$ or just by $S(G)$ 
when there is no room for confusion. As in the semigroup case $S(G,G_{+},u)$ is a compact convex space.
For detailed discussion on partially ordered abeilan groups and their
states see \cite{G}.

\subsection{Dimension functions, Lower semicontinuous dimension functions and
Quasitraces of a $C^{*}$-algebra.}

\begin{defn}\label{dmf}(\cite[Definition I.1.2]{BH} c.f \cite{C2}) 
A dimension function on a unital $C^{*}$-algebra $A$ is a function
$d:M_{\infty}(A)_{+}\to[0,\infty)$ which satisfies the following
conditions;

1. $d(1_A)=1$.

2. $d(a+b)=d(a)+d(b)$ for all $a,b\in M_{\infty}(A)_{+}$ with $a\perp b$.

3. $d(a)\leq d(b)$ for all $a,b\in M_{\infty}(A)_{+}$ with $a\preccurlyeq b$.
\end{defn}

\begin{rem*}In \cite{BH} dimension functions are defined on all
elements in $M_{\infty}(A)$ with some additional requirements. Its easily seen that the two definitions are equivalent and by replacing (1) above with the condition $\sup\lbrace d(a):a\in A_{+}\rbrace=1$ the definition extends to non
unital algebras.
\end{rem*}

The set of all dimension functions on $A$ is denoted by $DF(A)$.

\begin{lem}\label{dmfandstates}(\cite{BH} c.f \cite{C2}) 
For a unital $C^{*}$-algebra $A$, $DF(A)$ is in bijective correspondence with $S(K_{0}^{*}(A),K_{0}^{*}(A)_{_{++}},\gamma(\langle1_{A}\rangle))$ and
thus with  $S=S(W(A),\leq,\langle1_{A}\rangle)$.
\end{lem}

\begin{proof}
We outline the identifications involved.

Any $s\in S(W(A))$ uniquely determines a state $s^{\prime}$
on $K_{0}^{*}(A)$ which is given by $s^{\prime}(\gamma(x)-\gamma(y))=s(x)-s(y)$. Conversely
if $s^{\prime}\in S(K_{0}^{*}(A))$ then $s(x)=s^{\prime}(\gamma(x))$ is a state
on $W(A)$. The map $s\mapsto s^\prime$ sets up a natural affine homeomorphism between $S(W(A))$ and $S(K_{0}^*(A))$.

On the other hand if $f\in DF(A)$ then $s_f(\langle a \rangle)=f(a)$ is a state on $W(A)$ and if $s\in S(W(A))$ then $f_s$ defined on $M_\infty(A)_+$ by $f_s(a)=s(\langle a\rangle)$ is in $DF(A)$.
\end{proof}

\noindent We will use the above identification freely.

Given $a\in M_{\infty}(A)_+$ and $\epsilon>0$, by $(a-\epsilon)_{+}$ we
denote the element of $C^{*}(a)$ which corresponds (via the functional
calculus of $a$) to the function 
\[
f_{\epsilon}(t)=\max\lbrace t-\epsilon,0\rbrace,\, t\in\sigma(a),
\]
where $\sigma(a)$ is the spectrum of $a.$

A dimension function $s$ is said to be lower semicontinuous 
if for each $a\in M_{\infty}(A)_{+}$
\[
s(\langle a\rangle)\leq\liminf_{n}s_{n}(\langle a_{n}\rangle)
\]
whenever $(a_{n})$ is a sequence in $M_{\infty}(A)_{+}$ converging
to $a$ in norm. 
The above is equivalent to the requirement;
\[
s(\langle a\rangle)=\sup_{\epsilon>0}s(\langle(a-\epsilon)_{+}\rangle)
\] for all $a\in M_{\infty}(A)_{+}.$ 

The set of all lower semicontinuous dimension functions of $A$ is denoted by $LDF(A)$.

Recall that a quasitrace \cite[Definition II.1.1]{BH} is a complex-valued function on a $C^*$-algebra having all the 
usual properties of a tracial state, but with linearity assumed only on
commutative $C^*$-subalgebras. A 2-quasitrace is a  quasitrace on $A$ that
extends to $M_2(A)$. From \cite[Proposition II.4.1]{BH},
any $2$-quasitrace extends to $M_n(A)$ for all $n\in\mathbb{N}$.

A quasitrace $\tau$ is said to be normalized if $||\tau||=\sup\lbrace\tau(a):a\in A_{+},||a||\leq1\rbrace=1$.
In the case that $A$ is unital this is equivalent to $\tau(1_A)=1$.

As usual $QT(A)$ denotes the set of all normalized quasitraces of $A.$

Given $\tau\in QT(A)$ define $d_{\tau}:W(A)\to[0,\infty)$ by 
\[
d_{\tau}(\langle a\rangle)=\lim_{n\to\infty}(a^{\nicefrac{1}{n}}).
\]
 
\begin{thm}\cite[Theorems II.2.2 and II.3.1]{BH}\label{qtldf}
Let $A$ be a $C^{*}$-algebra. The map $d_{\tau}$ is a lower semicontinuous dimension
function on $A$ for each $\tau\in QT(A)$. The assignment $\tau\mapsto d_{\tau}$
gives an affine bijection from $QT(A)$ onto $LDF(A)$ which has a
continuous inverse with respect to the pointwise topologies on both
ends.
\end{thm}

For unital $A$, $QT(A)$ is a Choquet simplex \cite[Theorem II.4.4]{BH}. Given a Choquet
simplex $K$, the set of all non negative valued bounded lower semicontinuous
affine maps from $K$ into $\mathbb{R}$ is denoted $\mathit{LAff}_{b}(K)^{+}$. 
With pointwise addition and pointwise ordering $\mathit{LAff}_{b}(K)^{+}$ form
a ordered semigroup - see \cite{G} for a detailed discussion on these topics.

For $\langle a\rangle\in W(A)$ define $\iota(\langle a\rangle):QT(A)\to[0,\infty)$
by 
\[
\iota(\langle a\rangle)(\tau)=d_{\tau}(\langle a\rangle),\text{ }\forall\tau\in QT(A).
\]
Clearly $\iota(\langle a\rangle)$ is well defined with and $\iota(\langle a\rangle)\in\mathit{LAff}_{b}(QT(A))^{+}$ for all $\langle a\rangle\in W(A)$. Note that
$\iota$ defines a morphism from $(W(A),\leq,\langle 1_A\rangle)$  to 
$(\mathit{LAff}_{b}(QT(A))^{+},\leq,1)$ where $1$ is the constant
function $1$ on $QT(A)$ which is an order unit for $\mathit{LAff}_{b}(QT(A))^{+}.$

We end this section by recalling the definition of radius of comparison
of a $C^{*}$- algebra.

\begin{defn}
Let $A$ be a $C^{*}$-algebra. $A$
has finite radius of comparison if there is some real number $r>0$
such that the following hold for all $a,b\in M_\infty(A)_+$;
\begin{equation}\label{rc}
(d_{\tau}(\langle a\rangle)+ r\leq d_{\tau}(\langle b\rangle),\forall\tau\in QT(A))\implies a\preccurlyeq b.
\end{equation}

If $A$ is of finite radius of comparison, the radius of
comparison of $A$ ($rc(A)$) is the infimum of all $r$ as in\ref{rc}. If not the radius of comparison is infinite and we write
$rc(A)=\infty$. Note that $rc(A)=0$ iff $A$ has strict comparison.
\end{defn}

\section{\noindent Proof of the main results}

Unless stated otherwise all $C^*$-algebras
are assumed to be unital and stably finite. Recall that in the non stably finite case
conjectures hold trivially.

\begin{thm}\label{altBH} 
Let $A$ be a unital $C^{*}$-algebra.
Then $LDF(A)$ is dense in $DF(A)$ if and only if $\iota:(W(A),\leq,\langle1_A\rangle)\to (\mathit{LAff}_{b}(QT(A))^{+},\leq,1)$
is a stable order embedding, where $\mathit{LAff}_{b}(QT(A))^{+}$ and $\iota$ is as
defined in the previous section and $1$ denote the constant function $1$.
\end{thm}

\begin{proof}
Suppose $\iota$ is a stable order embedding. 

Let $K$ denote the closure (in pointwise
convergence) of $LDF(A)$ in $DF(A)$. Then $K$ is a compact convex
subset of $DF(A)$. Suppose $x,y\in W(A)$ are such that $d(x)<d(y)$
for all $d\in K$. The function defined on $K$ given by $d\mapsto d(x)-d(y)$
is strictly positive and continuous on $K$ in pointwise topology.
Since $K$ is compact the function attains a minimum $\delta>0$ on $K$.

Choose some $n\in\mathbb{N}$ large enough so that $n\delta>1$. 

Then,
\[
nd(x)+1\leq nd(y),\forall d\in K.
\]
In particular, 
\[
nd_{\tau}(x)+1\leq nd_{\tau}(y),\forall\tau\in QT(A).
\]

Therefore, 
\[
n\iota(x)+1\leq n\iota(y).
\]

Hence, as $\iota$ is a stable order embedding, there is some $m\in\mathbb{N}$
and $z\in W(A)$ such that, 
\[
mnx+z+\langle1_{A}\rangle\leq mny+z.
\]

Then it follows, 
\[
s(x)<s(y),\forall s\in DF(A).
\]
Therefore, by Lemma \ref{density} $K=DF(A)$, i.e. $LDF(A)$
is dense in $DF(A)$.

Now suppose $LDF(A)$ is dense in $DF(A)$.

Note that in general $\iota$ is an order preserving homomorphism. To verify its a stable order
embedding let $x,y\in W(A)$ and suppose that there is some $n\in\mathbb{N}$
such that, 
\[
n\iota(x)+1\leq n\iota(y).
\]

Then for all $\tau\in QT(A)$, 
\[
d_{\tau}(nx+\langle1_{A}\rangle)\leq d_{\tau}(ny).
\]

Therefore, since $LDF(A)$ is dense in $DF(A)$ by assumption, 
\begin{eqnarray*}
s(nx+\langle1_{A}\rangle)&\leq &s(ny),\forall s\in DF(A)\\
s(nx)&< & s(ny),\forall s\in DF(A).\\
\end{eqnarray*}
Therefore by Lemma \ref{domination}, there is some
$m\in\mathbb{N}$ and $z\in W(A)$ such that,
\[
mnx+\langle1_{A}\rangle+z\leq mny+z
\]
and $\iota$ is a stable order embedding.
\end{proof}

Given $s\in DF(A)$ define $\overline{s}:W(A)\to [0,\infty)$ by
\[
\overline{s}(\langle a\rangle)=\sup_{\epsilon>0} s(\langle(a-\epsilon)_+\rangle).
\]
We need the following Proposition from \cite{R} (c.f \cite{BH}) to prove part 2 of Theorem \ref{mainB}.

\begin{prop}\cite[Proposition 4.1]{R}\label{dftoldf}
Let $A$ be a $C^*$-algebra and let $s\in DF(A)$. Then $\overline{s}$ defined above is a well defined lower semicontinuous dimension function and $\overline{s}(\langle a\rangle)\leq s(\langle a\rangle)$ for all $a\in M_\infty(A)_+$. 
\end{prop}

The following is mainly a consequence of Krein-Milman Theorem.

\begin{lem}\label{lemc2}
Suppose $A$ is a unital $C^{*}$-algebra with $\partial_{e}(QT(A))$
finite and non empty . Then $LDF(A)$ is compact and moreover the
map $g:QT(A)\to LDF(A)$ given by $\tau\mapsto d_{\tau}$ is an affine
homeomorphism. If its also the case that Conjecture 1.1 holds for $A$
then $DF(A)=LDF(A)$ and $DF(A)$ is affinely homeomorphic to $QT(A)$.
\end{lem}

\begin{proof}

From Theorem \ref{qtldf}, $g:QT(A)\to LDF(A)$ is an affine bijection and $g^{-1}$
is continuous. Since $QT(A)$ is compact and convex, by Krein\textendash Milman Theorem $QT(A)$
is the closure of the convex hull of $\partial_{e}(QT(A))$.
As $\partial_{e}(QT(A))$ is assumed to be finite, its convex hull $co(\partial_{e}(QT(A)))$
is compact and therefore we in fact have;
\[ QT(A)=\overline{co(\partial_{e}(QT(A)))}=co(\partial_{e}(QT(A))).
\]
Thus, 
\[ LDF(A)=g(QT(A))=g(co(\partial_{e}(QT(A))))
\]
and on the other hand since
$g$ is an affine bijection $g(\partial_{e}(QT(A)))=\partial_{e}(LDF(A))$.
Therefore,  
\[LDF(A)=g(co(\partial_{e}(QT(A))))=co(\partial_{e}(LDF(A))).
\]
In particular, since $\partial_{e}(LDF(A))=g(co(\partial_{e}(QT(A))))$ is a non empty finite
set, $LDF(A)$ is compact and $g$ is a homeomorphism.

Now if Conjecture 1.1 is true then $DF(A)=\overline{LDF(A)}$. As we had just noted,
$LDF(A)$ is compact and so it is equal to its own closure. Thus $DF(A)=LDF(A)$ and 
is affinely homeomorphic to $QT(A)$ from the preceding paragraph.
\end{proof}

\begin{thm}\label{mainB}
 Let $A$ be any unital $C^{*}$-algebra.
The following hold.

1. If $A$ has finite radius of comparison then $LDF(A)$ is dense
in $DF(A)$.

2. If $W(A)$ almost unperforated then $LDF(A)$ is dense in $DF(A)$.

3. If $\partial_{e}(QT(A))$ is a finite set and if either of the
assumptions above (in 1,2) holds for $A$ then $DF(A)=LDF(A)$ and $DF(A)$
is affiinely homeomorphic to $QT(A)$. In particular $DF(A)$ is a
Choquet simplex.

\end{thm}

\begin{proof}
Proof of 1:

Let $rc(A)=r<\infty$. By Theorem \ref{altBH}
we only have to show that $\iota$ is a stable order embedding.

Let $x,y\in W(A)$ and suppose that there is some $n\in\mathbb{N}$
such that, 
\[
n\iota(x)+1\leq n\iota(y).
\]

Then for all $\tau\in QT(A)$, 
\[
d_{\tau}(nx+\langle1_{A}\rangle)\leq d_{\tau}(ny).
\]

Choose some $m\in\mathbb{N}$ large enough so that $m>r+1$. 

Then for all $\tau\in QT(A)$, 
\begin{eqnarray*}
d_{\tau}(mnx+\langle1_{A}\rangle)+r & = & d_{\tau}(mnx)+1+r\\
 & < & d_{\tau}(mnx)+m\\
 & = & md_{\tau}(nx+\langle1_{A}\rangle)\\
 & \leq & md_{\tau}(ny)\\
\end{eqnarray*}

Therefore, since $rc(A)=r$, 
\[
mnx+\langle1_{A}\rangle\leq mny
\]
and $\iota$ is a stable order embedding.

Proof of 2:

Again we only have to show that $\iota$ is a stable order embedding.

So let $a,b\in M_\infty(A)_+$ and suppose that there is some $n\in\mathbb{N}$ such that,
\[
\iota(\langle a\rangle)+1\leq n\iota(\langle b\rangle)
\]
Then for all $\tau\in QT(A)$, 
\begin{equation}
2nd_{\tau}(\langle a\rangle)+ 1< 2nd_{\tau}(\langle b\rangle)\label{eq:2}
\end{equation}

Fix $\epsilon >0$ and let $s\in DF(A)$ be arbitrary.
Then $\overline{s}\in LDF(A)$, where $\overline{s}$ is as in Proposition \ref{dftoldf}.

Thus, by equation \eqref{eq:2}
\begin{equation}
2n\overline{s}(\langle a\rangle)+1< 2n\overline{s}(\langle b\rangle).\label{eq:2.1}
\end{equation}

Note that by definition of $\overline{s}$ we have,
\begin{equation*}
s(\langle (a-\epsilon)_+\rangle)\leq \overline{s}(\langle a\rangle).
\end{equation*}

Combining this with \eqref{eq:2.1} we have,
\begin{eqnarray*}
s(2n\langle (a-\epsilon)_+\rangle+2\langle 1_A\rangle)&=&2 ns(\langle(a-\epsilon)_{+}\rangle)+2\\
 & \leq & 2n\overline{s}(\langle a\rangle)+1+1\\
 & < & 2n\overline{s}(\langle b\rangle)+1\\
 & \leq & 2ns(\langle b\rangle)+1\\
& =& s(2n\langle b\rangle+\langle1_A\rangle)
\end{eqnarray*}

Since $2n\langle b\rangle + \langle 1_A\rangle$ is an order unit for $W(A)$ and $s$ is arbitrary, we apply Proposition \ref{aup} to conclude
\[2n\langle (a-\epsilon)_+\rangle + 2\langle 1_A\rangle \leq 2n\langle b\rangle + \langle 1_A\rangle.
\]

Note that $\epsilon$ is arbitrary and in particular does not depend on $n$. 

Thus by \cite[Proposition 2.4]{R} it follows that,
$$2n\langle a \rangle + 2\langle 1_A\rangle\leq 2n\langle b\rangle +\langle 1_A\rangle.$$

In particular for $z=\langle 1_A\rangle\in W(A)$, $$2n\langle a\rangle +\langle 1_A\rangle + z\leq 2n\langle b\rangle + z$$ 
and we conclude that $\iota :W(A)\to LAff_b(QT(A))^{+}$ is a stable order embedding. This complete the proof of 2.

Proof of 3:
First part follows directly from Lemma \ref{lemc2} and parts 1 and 2 above. To see that $DF(A)$ is a Choquet 
simplex recall $QT(A)$ is Choquet.

\end{proof}

\end{document}